\documentclass[12pt]{amsart}

\usepackage{amssymb}

\usepackage{enumitem}

\usepackage{graphicx}

\makeatletter
\@namedef{subjclassname@2020}{%
  \textup{2020} Mathematics Subject Classification}
\makeatother

\usepackage[T1]{fontenc}

\usepackage{latexsym,graphicx}
\usepackage{float} 
\usepackage{pgfplots}
\usepackage{hyphenat}

\newtheorem{theorem}{Theorem}
[section]
\newtheorem{corollary}{Corollary}[section]
\newtheorem{lemma}{Lemma}[section]

\newtheorem*{ThV}{Varadarajan's Theorem}

\theoremstyle{definition}

\newcommand{\eps}{\varepsilon}

\newtheorem*{rem}{Remark}

\numberwithin{equation}{section}

\pgfplotsset{compat=1.15} 
\pgfplotsset{soldot/.style={only marks,mark=*, line width=0.2pt, mark size=1pt}}
 
\begin{document}


\baselineskip=17pt


\title[Uniqueness sets with angular density]{Uniqueness sets with angular density for spaces of entire functions, I. Basics }

\author[A. Kononova]{Anna Kononova} 

\address{\parbox{\textwidth}{
School of Mathematical Sciences\\
Holon Institute of Technology\\
Holon 5810201, Israel
}}

\email{anya.kononova@gmail.com}

\date{\today}

\begin{abstract}
    This is the first part of our work which is devoted to the uniqueness sets for spaces of entire functions.  In this part we consider a set $\Lambda$ with angular density with respect to the order $\rho>0,$ satisfying the Lindel\"of condition. We find the value of the critical zero set type for $\Lambda$  in geometrical terms. 
We give a necessary and sufficient condition for the coincidence of the critical zero set type and the critical uniqueness set type.     
    
    At the end of the paper we present an application of our results to random zero sets in Fock-type spaces.
    
\end{abstract}


\subjclass[2020]{Primary 30D15.}

\keywords{entire functions, zero distribution, uniqueness set}

\maketitle

\section{Introduction and main results}
This is the first part of our work which consists of three parts. These parts are essentially independent of each other.

Let $ \mathcal E_{\rho, \sigma}$ be the class of entire functions of order $\rho$ and of type not exceeding  $\sigma$. The latter means that, for each $\eps>0$ exists $C_\eps$ such that 
$$|f(z)|\le C_\eps e^{(\sigma+\eps)|z|^\rho}, \;\;\; z\in\mathbb C.$$

Let $\Lambda\subset \mathbb C $ be a discrete set,  and $\mathcal A$ be some class of entire functions. We say that $\Lambda$ is a {\bf uniqueness set for $\mathcal A$}, if
$$ f\in \mathcal A, \left.f\right|_\Lambda=0\;\;\Rightarrow\;\;f\equiv0,$$
where a zero of multiplicity $m$ is counted $m$ times.

{\it
Given $\rho>0$ and a discrete set $\Lambda\subset \mathbb C$, we let
\begin{align*}
    \sigma_U(\Lambda)
    =&\sup\{\sigma:\Lambda \text{ \;is\; a \;uniqueness\; set \;for\; }\mathcal E_{\rho, \sigma}\}\\
    =&\inf\{\sigma: \exists f\in \mathcal E_{\rho, \sigma}\setminus\{0\}: \;f|_\Lambda=0
    \}.
    \end{align*}
We call this quantity the {\bf critical uniqueness type of $\Lambda$}.
}

One of the reasons for the importance of this quantity is that, by a classical Markushevich duality  argument   \cite[Lecture\ 3, Sect.~4]{LL}, for $\rho=1$, the uniqueness of $\Lambda$ in $\mathcal E_{1, \sigma}$ is equivalent to the completeness of the system    of exponential functions $ E_\Lambda=\{e^{\overline \lambda z}:\lambda\in\Lambda\}$ in the space of holomorphic functions in the disk  $\sigma \mathbb D=\{\zeta:|\zeta|<\sigma$\} with topology of locally uniform convergence.

The importance of the case $\rho=2$ is due to the fact that the exponential functions $e_w(z)=e^{z\overline w}$ are the reproducing kernels of the classical Fock-Bargmann space $\mathcal B$ of entire functions  satisfying 
$$\|f\|^2=\frac{1}{\pi}\int_{\mathbb C}|f(z)|^2 e^{-|z|^2 }{\rm d}m(z)<\infty,$$
where $m$ is the planar Lebesgue measure. That is, $f(w)=\langle f,e_w\rangle$ for $f\in\mathcal B.$ In this case, $E_\Lambda$ is complete in the Fock-Bargmann space provided that  $\sigma_U(\Lambda)>1/2,$ and incomplete for $\sigma_U(\Lambda)<1/2.$

At present, in spite of many efforts \cite{Azarin, Azarin-Giner, Khabibullin1}, explicit expressions, which would allow one to {\it compute} $\sigma_U(\Lambda),$ are known only in  very few cases. In this work we consider this problem only for regularly distributed sets $\Lambda$.

Note that our interest in this problem arose from the study of the uniqueness property in Fock-type spaces of a random set of points, which will be discussed in Section \ref{Fock}.

\begin{rem}
    Given $p> 0$ and a set $M\subset [0,2\pi p],$ we always consider it as a subset of the factor space $\mathbb R/ (2\pi p\mathbb Z).$
\end{rem}

We say that the set $\Lambda$ has an angular density (with respect to the order $\rho$), if for all $0\le\alpha<\beta\le 2\pi,$ except, maybe, countably many values, there exists the limit
$$\Delta_\Lambda(\alpha, \beta)=\lim_{R\to\infty}\frac{n_\Lambda(R;\alpha,\beta)}{R^\rho},$$
where  $n_\Lambda(R;\alpha, \beta)=\#\{\lambda_k\in\Lambda:|\lambda_k|< R, \arg \lambda_k\in(\alpha,\beta)\}.$
The quantity $\Delta_\Lambda(\alpha, \beta)$ is called the {\bf angular density} of $\Lambda.$ We will treat it as a non-negative measure on $[0,2\pi],$ which will also be denoted by $\Delta_\Lambda.$

A set $\Lambda$ is called {\bf $\rho$-regular} if 
\begin{itemize}
    \item [ (i)] it has an angular density with respect to the order $\rho$;
    \item[(ii)] if $\rho$ is an integer, then in addition, the Lindel\"of condition holds:
    \begin{equation}    \label{Lind}        \lim_{R\to\infty}\sum_{0<|\lambda|\le R}\frac1{\lambda^{\rho}}\;\;\;{\rm exists\; and\;is\;finite}.\end{equation}
\end{itemize}

    Note that condition (ii) yields that the measure $\Delta_\Lambda$ has zero $\rho$-th moment
  \begin{equation}\label{moment}
      \int_0^{2\pi} e^{i\rho \theta} {\rm d} \Delta_\Lambda (\theta)=0\end{equation}
(otherwise, the sum in (\ref{Lind}) would grow as $\log R$ for $R\to\infty$).  

On the other hand, given a finite measure $\Delta$ satisfying \eqref{moment}, we can always construct \cite[Chapter II, Section 4]{Levin} a $\rho$-regular set $\Lambda$ such that $\Delta_\Lambda=\Delta$. 	  

For $\rho$-regular sets $\Lambda$, the Levin-Pfluger theory of entire functions of completely regular growth \cite{Levin} (for a streamlined approach to that theory see~\cite{Azarin79}) reduces the problem of determining the value of $\sigma_U(\Lambda)$ to a question about $\rho$-trigonometrically convex functions.

In this part, we always assume that the set $\Lambda$ is $\rho$-regular.

A function $h:[\alpha,\beta]\to\mathbb R$ is called {\bf $\rho$-trigonometrically convex} if, for any $\alpha\le\theta_1\le\theta\le\theta_2\le\beta$, $\;\theta_2-\theta_1<\frac\pi \rho,$ we have
$$h(\theta)\le \frac{h(\theta_1)\sin\rho(\theta_2-\theta)+h(\theta_2)\sin\rho(\theta-\theta_1)}{\sin\rho(\theta_2-\theta_1)}.$$

By $TC_\rho$ we denote the class of $\rho$-trigonometrically convex functions. For integer~$\rho,$
$$T_\rho=\{A\cos\rho\theta+B\sin\rho\theta: A,B\in\mathbb R\}\subset TC_\rho$$
is a subclass of   $\rho$-trigonometric functions.

For fundamental properties of $\rho$-trigonometrically convex functions see, for instance, \cite[Chapter\ I, Sect.~16]{Levin} and \cite[Lecture\ 8]{LL}.

With each measure $\Delta$ satisfying the property \eqref{moment}, we associate a $2\pi$-periodic $\rho$-trigonometrically convex function $h_\Delta$ such that 
\begin{equation}
\label{h''}h_\Delta''+\rho^2 h_\Delta=2\pi\rho\Delta
\end{equation}
in the sense of distribution. For a non-integer $\rho$, the function $h_\Delta$ is uniquely defined, while for an integer $\rho$, the function $h_\Delta$ is defined up to a $\rho$\hyp{trigonometric} term $k\in T_\rho$. Note that in the latter case the moment condition (\ref{moment}) guarantees $2\pi$-periodicity of $h_\Delta.$

{
There are explicit expressions for $h_\Delta$ in terms of $\Delta$  \cite[Chapter II, Theorem 1, Theorem 2]{Levin}. Given a measure $\Delta$, we define $h_\Delta$ as follows:
$$
\hspace{0.5cm}
h_\Delta(\theta)=
\left\{
\begin{array}{lr}\displaystyle
    \frac{\pi}{\sin\pi\rho}\int_{\theta-2\pi}^\theta\cos\rho(\theta-\varphi-\pi){\rm d}\Delta(\varphi), & \;\;\;  \rho {\notin\mathbb N},\hspace{2cm}
\vspace{3pt}
    \\
      \displaystyle -\int_{\theta-2\pi}^\theta{(\varphi-\theta)}\sin\rho(\varphi-\theta){\rm d}\Delta(\varphi),&  \;\;\;\rho {\in\mathbb N.}\hspace{2cm}
    \end{array}
\right.
$$

Note that, unlike the formulas from \cite{Levin}, we  omit the $\rho$\hyp{trigonometric} term in the case $\rho\in\mathbb N$, so that for each $A,\tau\in\mathbb R,$ the function $h_\Delta(t)+A\cos\rho(t-\tau)$ also corresponds to the measure $\Delta$ and satisfies the condition \eqref{h''}.

}

Given a $\rho$-regular set $\Lambda$ with angular distribution $\Delta_\Lambda$
 we also use the notation $h_\Lambda=h_{\Delta_\Lambda}$.

We call the set $\Lambda$ the {\bf zero set} for some class  of analytic functions $\mathcal A$ if there exists a function $f\in \mathcal A $ such that $f^{-1}\{0\}=\Lambda.$

{\it Given $\rho>0$ and a discrete set $\Lambda\subset \mathbb C$, we define the {\bf critical zero set type} $\sigma_Z(\Lambda)$ by
$$\sigma_Z(\Lambda)=\inf\left\{\sigma: \exists f\in \mathcal E_{\rho,\sigma}:\;f^{-1}\{0\}=\Lambda\right\}.$$
}

Clearly,
\begin{equation}\label{sigma}
    \sigma_U(\Lambda)\le\sigma_Z(\Lambda).\end{equation}
Unlike $ \sigma_U(\Lambda),$ the quantity  $\sigma_Z(\Lambda)$ is easy to compute.

The following Theorem \ref{T1} is a  consequence of the Levin-Pfluger theory of entire functions of completely regular growth \cite{Levin}.
\renewcommand{\theenumi}{(\roman{enumi})} 
\renewcommand{\labelenumi}{\theenumi} 

\begin{theorem}
    \label{T1}
    Suppose $\Lambda$ is a $\rho$-regular set with angular density $\Delta$.
    \begin{enumerate}
        \item [ (i)] If $\rho$ is non-integer, then
        $$\sigma_Z(\Lambda)=\max_{t\in[0,2\pi]}h_\Delta(t).$$
        \item [ (ii)] If $\rho$ is integer, then
        \begin{equation}\label{maxmin}
            \sigma_Z(\Lambda)=\min_{k\in T_\rho}\max_{t\in[0,2\pi]}[h_\Delta(t)+k(t)].
        \end{equation}
    \end{enumerate}
\end{theorem}

For $\rho=1,$ the right-hand side of (\ref{maxmin}) has a simple geometric meaning \cite{Azarin-Giner, Khabibullin2}.
In this case,  a trigonometrically convex function $h_\Delta$ is a supporting function of a planar convex  compact set $I_\Delta$. Adding a trigonometric function 
$$A\cos\theta+B\sin\theta={\rm Re} (Ce^{i\theta}),$$
we shift $I_\Delta$ by $Ce^{i\theta}$. Thus, in this case, the right-hand side of (\ref{maxmin}) equals the circumradius $R_\Delta$ of the  set $I_\Delta$, while the needed translation vector $Ce^{i\theta}$ is uniquely  characterized by the condition that the origin lies in the convex hull of the set 
$$\{e^{i\theta}:\;\;h_\Delta(\theta)+{\rm Re}(Ce^{i\theta})=R_\Delta\}.$$

A similar interpretation of the right-hand side of (\ref{maxmin}) holds for all integer values of $\rho.$
Given $h\in TC_\rho$, 
we associate $\Delta$
with it via \eqref{h''}.  Assuming that $\rho$ 
is integer, we let
$$h^*(\theta)=\max_{j=0,1,\ldots, \rho-1} h\left(\frac {\theta+2\pi j}{\rho}\right),$$
and observe that the function $h^*$ is $1$-trigonometrically convex. By $I_\Delta^*$ we denote the planar convex compact set, having supporting function $h^*$. By $R_\Delta^*$ we denote the circumradius of $I_\Delta^*.$

Set $M_h=
\{\theta:\;\;h(\theta)=\displaystyle
\max_{t\in[0,2\pi]}h(t)\}$. 

By ${\rm conv}X,$ we denote the convex hull of a plane set X.

\begin{theorem}
\label{T2}
    Let $\rho$ be a positive integer, and let $h\in TC_\rho.$

    \begin{enumerate}
        \item [A.] The following two conditions are equivalent:
        \begin{enumerate}
            \item[(i)] $0\in { \rm conv}\{e^{it\rho}: t\in M_h\};$
            \item [(ii)]$\displaystyle\min_{k\in T_\rho}\;\;\max_{t\in[0,2\pi]}\bigl(h(t)+k(t)\bigr)=\max_{t\in[0,2\pi]} h(t).$
        \end{enumerate}
        \item [B.] $\displaystyle\min_{k\in T_\rho}\max_{t\in[0,2\pi]}\bigl(h(t)+k(t)\bigr)=R_\Delta^*.$
    \end{enumerate}
\end{theorem}

We call a function $h\in TC_\rho$ ($\rho$ is an integer) {\bf $\rho$-balanced}, if it satisfies condition {\it(i)}. Note that any function $h\in TC_\rho$ can be made $\rho$-balanced by adding a $\rho$-trigonometric function.

\begin{corollary}\label{Cor1}
    Suppose $\rho\in\mathbb N$, and $\Lambda$ is a $\rho$-regular set with angular density~$\Delta$.
   Then $$\sigma_Z(\Lambda)=R^*_\Delta=\max_{t\in[0,2\pi]}\bigl(\widehat h(t)\bigr),$$
   where $ \widehat h$ is  the $\rho$-balanced function such that $ \widehat h-h\in T_\rho$.

   Moreover, 
   $$\sigma_Z(\Lambda)=\min\left\{\sigma: \exists f\in \mathcal E_{\rho,\sigma}:\;f^{-1}\{0\}=\Lambda\right\}.$$
\end{corollary}

Now, let us turn to the uniqueness sets. The next theorem is also an immediate consequence of the Levin-Pfluger theory:

\begin{theorem}\label{T3}
    Suppose $\Lambda$ is a $\rho$-regular set. Then 
    $$\sigma_U(\Lambda)=\inf_{k\in TC_\rho}\max_{t\in[0,2\pi]}\bigl(h_\Lambda(t)+k(t)\bigr).$$
    Furthermore for $\rho\le 1/2$ and for $\rho=1$,
    $$\sigma_U(\Lambda)=\sigma_Z(\Lambda).$$
\end{theorem}

It is not difficult to show that for other values of $\rho$ the quantities $\sigma_U(\Lambda)$ and $\sigma_Z(\Lambda)$ can be different:
\begin{theorem}\label{T4}
    Given $\rho\in (1/2, \infty)\setminus \{1\},$ there exists a $\rho$-regular set $\Lambda$ with $\sigma_U(\Lambda)<\sigma_Z(\Lambda)$.
\end{theorem}

Our next result characterizes the equality $\sigma_U(\Lambda)=\sigma_Z(\Lambda)$ in the case $\rho>1/2$.
Given a function $h\in TC_\rho$, we set
\begin{equation}
    \label{lb-m}
\widehat h=\begin{cases}
    h,&\rho \notin\mathbb N,\\
    \rho{\rm-balanced\;\; modification\;\; of \;}h,&\rho \in\mathbb N.
\end{cases}
\end{equation}

We call a set $M\subset [0,2\pi]$ {\bf locally $\rho$-balanced}, if there exists $\theta^*\in [0,2\pi]$ such that
\begin{equation}\label{loc-bal-geom}
0\in { \rm conv}\{e^{it\rho}: t\in M\cap [\theta^*,\theta^*+{2\pi}/{\rho}]\}.
\end{equation}

We call the function $h\in TC_\rho, \;\;\rho>1/2$, locally $\rho$-balanced according to the corresponding property of the set $M_{\widehat h}=\{\theta:\;\widehat h(\theta)=\displaystyle\max_{[0,2\pi]}\widehat h(t)\}$.

\begin{theorem}\label{T5}
    Let $\rho >1/2$, and let $\Lambda$ be a $\rho$-regular set. Then TFAE:
\begin{enumerate}
    \item [(i)] $\sigma_U(\Lambda)=\sigma_Z(\Lambda)$.
    \item [(ii)] the function $\widehat h_\Lambda$ is locally $\rho$-balanced;

\end{enumerate}
\end{theorem}
Our next result gives sharp lower bounds for $\sigma_U(\Lambda)$, which might be useful in the case when $\sigma_U(\Lambda)<\sigma_Z(\Lambda)$.

\begin{theorem}\label{T6}
    Let $\rho>1/2$, and let $\Lambda$ be a $\rho$-regular set.

    Put $$A_\Lambda:=\displaystyle \frac12\max_{\theta\in[0,2\pi]}\left(h_\Lambda(\theta)+h_\Lambda(\theta+\frac\pi\rho)\right).$$
Then  \begin{equation}  \label{6+}\sigma_U(\Lambda)\ge A_\Lambda.
    \end{equation}

    In addition,
\begin{equation}\label{6ge}
    \sigma_U(\Lambda)\ge C
\end{equation}
provided that the set $
\{t:\widehat h_\Lambda(t)\ge C\}$ is locally $\rho$-balanced.
\end{theorem}

{ Our last two theorems  are inspired by a result of Ascensi, Lyubarskii and Seip  \cite[Theorem 2]{ALS}. Given a $\rho$-regular set $\Lambda\subset\mathbb C$  with an angular density $\Delta_\Lambda$,  these theorems  focus on exploring possible  densities 
$$\mathcal D(\Lambda)=\Delta_\Lambda([0,2\pi))$$
in some critical cases.

\begin{theorem}\label{T7}
Let $\mathcal{A}_\rho$ denote the class of all $\rho$-regular sets $\Lambda\subset\mathbb C$ that satisfy:
\begin{enumerate}
  \item $\Lambda$ is a zero set for $\mathcal E_{\rho,1}$;
  \item $\Lambda$ is a uniqueness set for every $\mathcal E_{\rho,\sigma}$ with $\sigma<1$.
\end{enumerate}
Consider the density map
\[
\mathcal D:\ \mathcal A_\rho\to\mathbb R_{+},\qquad 
\mathcal D(\Lambda)=\Delta_\Lambda([0,2\pi)),
\]
where $\Delta_\Lambda$ is the angular density measure of $\Lambda$.
Then:
\[
\mathcal D(\mathcal A_\rho)=
\begin{cases}
\displaystyle\left[\dfrac{\sin(\pi\rho)}{\pi},\ \rho\right], & \rho\in(0,\tfrac12],\\[2.2ex]
\displaystyle\left[\dfrac{1+|\cos(\pi\rho)|}{\pi},\ \rho\right], & \rho\in(\tfrac12,\infty).
\end{cases}
\]
\end{theorem}

{
    \begin{theorem}\label{ALS1}
    Let $\mathcal{B}_\rho$ denote the class of all $\rho$-regular sets $\Lambda\subset\mathbb C$ that satisfy:
  \begin{enumerate}

        \item $\Lambda$ is a non-uniqueness set for $\mathcal E_{\rho,1}$;
        \item  for every $\rho$-regular set $\Lambda_1$ with non-zero angular density $\Delta_1$, the set $\Lambda\cup\Lambda_1$ is a uniqueness set for  $\mathcal E_{\rho,1}$.
     \end{enumerate} 
Consider the density map
\[
\mathcal D:\ \mathcal B_\rho\to\mathbb R_{+},\qquad 
\mathcal D(\Lambda)=\Delta_\Lambda([0,2\pi)),
\]
where $\Delta_\Lambda$ is the angular density measure of $\Lambda$.
Then:
\[
\mathcal D(\mathcal B_\rho)=
\begin{cases}
\displaystyle\left[\frac{\sin\pi\rho}{\pi}, \rho\right], & \rho\in(0,\tfrac12],\\[2.2ex]
\displaystyle\left[\frac1\pi\left(\sin\frac{\pi\{2\rho\}}2+[2\rho]\right), \rho\right], & \rho\in(\tfrac12,\infty).
\end{cases}
\]

\end{theorem}
}
\begin{rem}
    Note that for the case $\rho=2$ our bounds in Theorem \ref{ALS1} coincide with those in \cite[Theorem 2]{ALS}, which is somewhat surprising, given the considerably more "massive" perturbation here.
\end{rem}

    At the end of the paper, in Section \ref{Fock}, we present an application of our results to random zero sets in Fock-type spaces.

    }

\section{Proofs of Theorems \ref{T1} and \ref{T3}}
By our conditions on $\Lambda$ and by Levin \cite[Chapter\ I, Sect.~10]{Levin} we can define an entire function~$W_\Lambda$:
\begin{itemize}
    \item in the case of $\rho\in\mathbb N$
    \begin{equation}   \label{W1}
W_\Lambda(z)=\exp\left(\displaystyle-\frac1\rho \sum_k\frac{z^\rho}{\lambda_k^{\rho}}\right)\prod_{\lambda_k\in\Lambda} G\left(\frac{z}{\lambda_k};\rho\right);
\end{equation}
 \item in the case of $\rho\notin\mathbb N$
    \begin{equation}   \label{W2}
W_\Lambda(z)=\prod_{\lambda_k\in\Lambda} G\left(\frac{z}{\lambda_k};[\rho]\right),
\end{equation}
where
$\displaystyle G(w; d) :=(1-w) \exp(w+\frac{w^2}{2}+\dots+\frac{w^d}{d}) . $
\end{itemize}

 Let  $f\in \mathcal E_{\rho,\sigma}$ be such that $f^{-1}\{0\}=\Lambda$.
By the Hadamard  factorization  theorem, $f$ can be represented as an Hadamard product
    $$f=e^{P_N}W_\Lambda,$$
where $P_N$ is a polynomial of degree $N\le[\rho]$.

Recall that the {\bf indicator} of an entire function $f$ of order 
$\rho$ is defined by
$$h_f(t):=\limsup_{r\to\infty}\frac{\log |f(re^{it})|}{r^\rho},\qquad t\in[0,2\pi).$$
Moreover, due to the Levin-Pfluger theory \cite[Chapter III, Sect.3, Theorem 4]{Levin}, if the zero set of $f$ is $\rho$-regular, there exists an exceptional set $E$, the union   of disks with zero linear density, such that  
\begin{gather*}
    h_f(t)=\lim_{\substack{
r\to\infty\\ re^{it}\notin E}}\frac{\log|f(re^{it})|}{r^\rho}.
\end{gather*}
One can think of the exceptional set $E$ as the union of small neighbourhoods of the zeros of $f$, where supremum in each disk is governed by the maximum principle. 

Furthermore, by \cite[Chap.~II, Sect.~1, Theorem~1, Theorem~2]{Levin}, outside $E$ the indicator $h_f$ satisfies the following relation:
\begin{gather*}
    h_f(t)=
\begin{cases}
    h_\Lambda(t),\;\;&\rho\notin\mathbb N,\\
    h_\Lambda(t) + a\cos\rho t+b\sin\rho t,\; a,b\in \mathbb R,\;\;&\rho\in\mathbb N.
\end{cases}\end{gather*}

If $\rho\notin \mathbb N$ the type of the function $f$ is uniquely determined by the set $\Lambda$ and  $ \sigma_Z(\Lambda)=\displaystyle\max_{[0,2\pi]}h_\Lambda(t).$

In the case of $\rho\in \mathbb N$, the type of the function also depends  on the exponential factor $e^{P_N(z)}$. Let us assume that $N=\rho$ (allowing the leading coefficient to be zero). Multiplying the function $W_\Lambda$ by an exponential factor $e^{(A-iB)z^{\rho}}$ we do not change its set of zeros, while its indicator changes by a $\rho$-trigonometric function $k(t)=A\cos\rho t+B\sin\rho t.$
Hence, for every $k\in T_\rho$ we obtain a function $\widetilde f\in\mathcal E_{\rho, \sigma} $ such that $\widetilde f^{-1}\{0\}=\Lambda$ and 
$h_{\widetilde f}=h_\Lambda+k$, varying the leading coefficient of the polynomial $P_N(z)$. Therefore, $$\sigma_Z(\Lambda)\le  \inf_{k\in T_\rho}\max_{t\in[0,2\pi]}
[h_\Lambda(t)+k(t)].$$ On the other hand, it follows from the compactness argument that there exists $k_{\Lambda}\in T_\rho$ such that 
$$\max_{t\in[0,2\pi]}
[h_\Lambda(t)+k_{\Lambda}(t)]  =\min_{k\in T_\rho}\max_{t\in[0,2\pi]}[h_\Lambda(t)+k(t)].$$
Consequently, we get  $$\sigma_Z(\Lambda)=  \min_{k\in T_{\rho}}\max_{t\in[0,2\pi]}
[h_\Lambda(t)+k(t)].$$

Theorem \ref{T1} is proved.\hfill $\Box$

We now proceed to the proof of Theorem \ref{T3}. 

Let $\rho>0$.  
By the definition, the set $\Lambda$ is {\it not} a uniqueness set for  $\mathcal E_{\rho,\sigma}$ if and only if there exists an entire function $g$ (a multiplicator), such that $g\cdot W_\Lambda\in \mathcal E_{\rho,\sigma}.$
Note, that if such a function exists, then it has a finite type with respect to the order $\rho$, and, by Levin theorem on the indicator of the product of two entire functions { \cite[Chapter\ III, Sect.~4, Theorem\ 5]{Levin}}, we have 
$$h_{g\cdot W_\Lambda}=h_g+h_{W_\Lambda}.$$
So, the existence of a function $g$, such that 
$$\max_{t\in[0,2\pi]}h_{g\cdot W_\Lambda}(t)\le\sigma,$$ is necessary and sufficient for the set $\Lambda$ to be a non-uniqueness set for $\mathcal E_{\rho,\sigma}$.
If we can find a $\rho$-trigonometrically convex function $k$
such that 
$$\max_{t\in[0,2\pi]}\bigl(k(t)+h_{W_\Lambda}(t)\bigr)\le\sigma,$$
 then we can take any function $g$ of completely regular growth with indicator $h_g=k$. Thus, we substitute the problem of finding a multiplier $g$ with the problem of finding a $\rho$-trigonometrically convex function $k$.
It follows that  $$\sigma_U(\Lambda)= \inf_{k\in TC_{\rho}} \max_{t\in[0,2\pi]}[h_\Lambda(t)+k(t)].$$

Note that for $\rho\le 1/2$ all $\rho$-trigonometrically convex functions are non-negative, so we  cannot lower the maximal value of the  function $h_\Lambda$ by adding a $\rho$-trigonometrically convex function. Hence, for $\rho\le 1/2$ we have
$$\sigma_U(\Lambda)=\sigma_Z(\Lambda).$$
Let us show that this equality also holds for the case $\rho =1.$
Suppose that $\sigma_U(\Lambda)<\sigma_Z(\Lambda)$. That is, $\exists k^*\in TC_{1}$ such that 
\begin{equation}
    \label{p=1}\max_{t\in[0,2\pi]}\bigl(h^*_{\Lambda}(t)+k^*(t)\bigr)=\sigma_U(\Lambda)<\sigma_Z(\Lambda)=  \max_{t\in[0,2\pi]}h^*_\Lambda(t).\end{equation}
where $h^*_{\Lambda}$ is the 1-balanced modification of $h_\Lambda$.

As a continuous function, $k^*$ has at least one point of maximum on the interval $[0,2\pi]$. Then, due to the basic properties of the  trigonometrically convex functions \cite[Chapter\ I, Sect.~16]{Levin} there is a closed interval $J$ of the length $\pi$ such that $k^*(t)\ge 0 \;\;\forall t\in J$. 

It follows from the definition of $1$-balanced function, that every closed subinterval of the interval $[0,2\pi]$ of the length $\pi$ contains at least one point of maximum of the function $h^*_\Lambda$, in particular, there exists $t_M\in J$: $$h^*_\Lambda(t_M)=\displaystyle\max_{t\in[0,2\pi]}h^*_\Lambda(t).$$ Then, by (\ref{p=1}), 
$$h^*_\Lambda(t_M)=\max_{t\in[0,2\pi]}h^*_\Lambda(t)>\max_{t\in[0,2\pi]}(h^*_\Lambda(t)+k^*(t))\ge h^*_\Lambda(t_M)+k^*(t_M)\ge h^*_\Lambda(t_M),$$
 and we get a contradiction.

 Theorem \ref{T3} is proved.\hfill $\Box$

\section{Proof of Theorem \ref{T2} and Corollary \ref{Cor1}}

We start with the proof of Theorem \ref{T2} and prove Corollary \ref{Cor1} at the end of this section.

\paragraph {\bf Part A}

{\it (i)$\Rightarrow$(ii)}

    Let $h\in TC_\rho$ be $\rho$-balanced. In other words, for any $\alpha\in\mathbb R$
    each of the sets
    $$A_1(\alpha):=\bigcup_{k\in\mathbb Z} \left[\alpha+\frac{2\pi k} \rho,\alpha+\frac{\pi  (2k+1)}\rho\right) $$
    and $$ A_2(\alpha):=\bigcup_{k\in\mathbb Z} \left[\alpha+\frac{\pi (2k-1)} \rho,\alpha+\frac{2\pi k}\rho\right) $$
    has a non-empty intersection with the set $M_h$.
    
     Suppose that condition {\it (ii)} does not hold, then there is a function $k\in T_\rho,$ $k(t)=C\sin \rho(t-t_0)$, where $C, t_0\in\mathbb R,$ such that $$\displaystyle\max_{t\in[0,2\pi]} (h(t)+k(t))<\max_{t\in[0,2\pi]} h(t).$$

        Consider now two points { $\alpha_1\in A_1(t_0)\cap M_h, \alpha_2\in A_2(t_0)\cap M_h,$}
       then
       $$k(\alpha_1)\cdot k(\alpha_2)=C^2\sin \rho(\alpha_1-t_0)\sin \rho(\alpha_2-t_0)\le 0.$$
       Therefore, 
                           either $h(\alpha_1)+k(\alpha_1)\ge h(\alpha_1)=\displaystyle\max_{t\in[0,2\pi]}  h(t),$
                            or
       $h(\alpha_2)+k(\alpha_2)\ge h(\alpha_2)=\displaystyle\max_{t\in[0,2\pi]}  h(t)$. 
       
       In either case, we arrive at a contradiction.
           
 {\it (ii)$\Rightarrow$(i)}
       
       Now, let us suppose that $h$ is not $\rho$-balanced, that is, the condition {\it(i)} is not satisfied. 
Then there exists an $\alpha$ such that 
$M_h\subset A_1(\alpha).$
Without loss of generality we can suppose that $\alpha=0, $
and hence $M_h\subset A_1(0)$.  
 Moreover, since $M_h$ is a compact set, we can suppose that for some $\eps>0$ 

$$M_h\subset A_1^\eps:=\bigcup_{k\in\mathbb Z} \left(\frac{2\pi k} \rho+\eps,\frac{\pi  (2k+1)}\rho-\eps\right) .$$
Define also $A_2^\eps:=\mathbb R\setminus A_1^\eps.$
Put
$$a:= \max_{t\in\mathbb R} h(t);\;\;\;b:= \sup_{A^\eps_2} h(t).$$
From the construction, it follows that $a>b$. 
Consider the function
$$\widetilde h(t):=h(t) - \frac{a-b}2\sin \rho t. $$
We have
\begin{itemize}
    \item if $t\in A_1^\eps,$ then $$\widetilde h(t)\le a-\frac{a-b}{2} \sin\rho\epsilon<a,$$
\item
if $t\in A_2^\eps$
$$\widetilde h(t)\le b+\frac{a-b}2=\frac{a+b}2<a.$$
\end{itemize}
So, the condition {\it (ii)} is not satisfied.

\paragraph {\bf Part B}
Recall the definition of the function $h^*$ given in the introduction:
$$h^*(t)=\max_{j=0,1,\ldots, \rho-1} h\left(\frac {t+2\pi j}{\rho}\right),$$
It is a ${2\pi}$-periodic 1-trigonometrically convex function that serves as a support function for the convex set $I^*$ with circumradius $R^*$ and circumcenter  $C^*\in\mathbb C$. By taking the geometrical sum  of the set $I^*$ and the point $-C^*$ (that is, shifting the set $I^*$ so that its center coincides with the origin), we obtain a convex set with  a supporting  function $h^*_0(t)=h^*(t)-{\rm Re} (C^* e^{it})$. 
 Note, that 
 $$h^*_0(t)=
    \max_{j=0,\dots,p-1}\left(\widetilde h\left(\frac{t+ 2\pi j}{\rho}\right)\right),
 $$
where
$$\widetilde h(t)=
    h(t)-{\rm Re} (C^* e^{ipt}).
$$
Furthermore, the function $\widetilde h$ is $\rho$\hyp{balanced}, since  the function $h^*_0$ is $1$\hyp{balanced}, and the set   $\{e^{i\rho t}: t\in M_{\widetilde h}\}$  coincides  with the set $\{e^{i t}: t\in M_{ h_0^*}\}$. Hence, by part A, we have
$$
\min_{k\in T_\rho}\max_{t\in [0,2\pi], }(h(t)+k(t))=\max_{t\in [0,2\pi]} h(t)=\max_{t\in [0,2\pi]}\widetilde h(t)=R^*.$$
\hfill $\Box$

\paragraph{\bf Example 1.}
Let $ \rho\in\mathbb N$.
Consider the following family of probability measures on the interval $[0,2\pi]$
$$\Delta_1^{n}:=\frac1n\sum_{j=1}^n\delta_{(2j+1)\pi/n},\;\;\;\;\;n\in\mathbb N, n\ge 2\rho,$$
and let $\Delta_1^{\infty}$ be the normalized Lebesgue measure on $[0,2\pi].$
Let $\Lambda_n$ and $\Lambda_\infty$ be  $\rho$-regular sets such that $\Delta_{\Lambda_n}=\Delta_1^{n}$ and $\Delta_{\Lambda_\infty}=\Delta_1^{\infty}.$
Then 
$$
 \displaystyle h_{\Delta_1^{n}}(t)=\frac{\pi }{n\sin\frac{\rho \pi}{n}}\max\left\{\cos \rho\bigl(t-\frac{2\pi j}{n}\bigr),\;\;j=0,\ldots,n-1\right\},$$
 and  
$h_{\Delta^{\infty}_1}(t)=\frac1\rho.$
It is easy to see that all these functions are $\rho$-balanced, hence,  we have
{
$$\sigma_Z(\Lambda_n)=\frac{\pi }{n\sin\frac{\rho \pi}{n}}.$$}
As $n$ tends to infinity and  the set $\Lambda$ spreads evenly over the interval $[0,2\pi]$, approaching the uniform distribution $\Delta^{\infty}_1$, we obtain 
$$\lim_{n\to\infty}\sigma_Z(\Lambda_n)=\frac1\rho=\sigma_Z(\Lambda_\infty).$$

We now proceed to the proof of Corollary \ref{Cor1}.
The first  statement of the corollary is an immediate consequence of Theorem \ref{T1} and Theorem \ref{T2}.

To show that in the definition of the critical zero set type we can replace infimum by minimum, note that the function $\widehat h$ is an indicator of the entire function $f_\Lambda(z)=e^{Cz^\rho}W_\Lambda(z)$  with an appropriate value of $C\in\mathbb C$.  
Since  $\widehat h(t)=h_{f_\Lambda}(t)\le R_{\Delta}^*=\sigma_Z(\Lambda),$ we have $f_\Lambda\in\mathcal E_{\rho, \sigma_Z(\Lambda)}$, while $f^{-1}\{0\}=~\Lambda.$

\hfill $\Box$

\section{Proof of Theorem \ref{T4}} 

Here, given $\rho\in (1/2,\infty)\setminus\{1\},$ we will provide some examples of non-negative measures $\Delta$ such that for any $\rho$-regular set  $\Lambda$ with $\Delta_\Lambda=\Delta$ we have
$\sigma_U(\Lambda)<\sigma_Z(\Lambda)$.

\paragraph{\bf Example 2.}

Let $\rho\in\mathbb N, \rho \ge 3.$
Put
$$\Delta_2:=\delta_{\pi/\rho}+\delta_{2\pi/\rho}+\delta_{4\pi/\rho}+\delta_{5\pi/\rho}.$$
Then (see  Fig.~\ref{graph 101})$$
h_{\Delta_2}(t)=\begin{cases}
    2\pi\cos \rho t,&t\in[\frac\pi\rho,\frac{2\pi}\rho]\cup[\frac{4\pi}\rho,\frac{5\pi}\rho];\\
    0,&{\rm elsewhere.}
\end{cases}
$$

\begin{figure}[H]
\begin{center}

{
\begin{tikzpicture}[declare function={
    g(\x)=
      \x< pi ? 0:
         (\x< 2*pi ? -sin(\x r): 
      (\x<  4*pi ? 0: 
 (\x< 5*pi ? sin(\x r):
      (\x< 8*pi ? 0:
      (\x<9*pi  ?  0:
      (\x<10*pi  ?  -sin(\x r):
     0)))));}]
    \begin{axis}[
      grid=both, 
      grid style={line width=0.1pt, draw=gray!75},
      axis lines=center,
      axis line style={black},
  unit vector ratio = 1 8,
     xmin=-1, xmax=32,
     ymin=-0.5, ymax=1.3,
      xtick=\empty,
      ytick=\empty,
      every axis plot/.append style={line width=1pt, color=black},
          ]
        \addplot[ thick, samples at={0,0.1,...,35}] {g(x)}; 
\fill(pi,0) circle (1pt) node[below] {\tiny $\displaystyle \frac{\pi}{\rho}$} ;
\fill(2*pi,0) circle (1pt) node[below] {\tiny $\displaystyle \frac{2\pi}{\rho}$} ;
\fill(3*pi,0) circle (1pt) node[below] {\tiny $\displaystyle \frac{3\pi}{\rho}$} ; 
\fill(4*pi,0) circle (1pt) node[below] {\tiny $\displaystyle \frac{4\pi}{\rho}$} ; 
\fill(5*pi,0) circle (1pt) node[below] {\tiny $\displaystyle \frac{5\pi}{\rho}$} ; 
\fill(6*pi,0) circle (1pt) node[below] {\tiny $\displaystyle \frac{6\pi}{\rho}$}; 
\fill(7*pi,0) circle (1pt) node[below] {\tiny $\displaystyle {\ldots}$}; 
\fill(8*pi,0) circle (1pt) node[below] {\tiny $\displaystyle {2\pi}$}; 
\fill(0,0) circle (1pt); 
\fill(0,1) circle (1pt) node[above] {\tiny $\displaystyle {\;\;\;\;\;\;2\pi}$}; 
\draw [dashed](0,1) - - (32,1) ;
   \end{axis}
        \end{tikzpicture}}
   \caption{Illustration to Example 2: $y=h_{\Delta_2}(t). $
   }\label{graph 101}

\end{center}   
\end{figure}
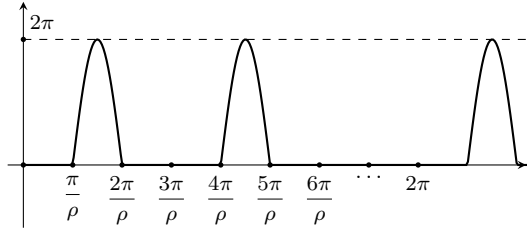

{ Note, that the function $h_{\Delta_2}$ is $\rho$-balanced, since
 $0\in{\rm conv}\{e^{\frac{3\pi}{2} i}, e^{\frac{9\pi}{2}i}\}.$}
Hence, $$\sigma_Z(\Lambda)=\max_{[0,2\pi]}\left( h_{\Delta_2}(t)\right)={2\pi}$$ for every $\rho$-regular set $\Lambda$ such that $\Delta_\Lambda=\Delta_2$. 

Let
$$k(t)=
    \begin{cases}\displaystyle
        {\pi}\sin\rho t, &\text{if } t\in[0,3\pi/\rho];\\
       \displaystyle{-\pi}\sin\rho t, &\text{if } t\in[3\pi/\rho, 6\pi/\rho].
    \end{cases}$$
Then $\displaystyle\max_{[0,2\pi]}(h_{\Delta_2}(t)+k(t))=\pi<\sigma_Z(\Lambda)$. Hence, 
{$$\sigma_U(\Lambda)\le \pi<\sigma_Z(\Lambda)=2\pi.$$
}

\paragraph{\bf Example 3.} Let $ \rho= 2.$
The construction from the previous example does not work for this case. Let
 $\Delta_3 =\delta_0+\delta_{2\pi/3}+\delta_{4\pi/3}.$
    Then (see Fig.~\ref{Picture1})
    $$h_{\Delta_3}(t)=
        \frac{2\pi}{\sqrt3}\cos 2(t-\pi/3), \text{if } t\in[0,2\pi/3],\;\;\;h_{\Delta_3}(t+2\pi/3)=h_{\Delta_3}(t).$$

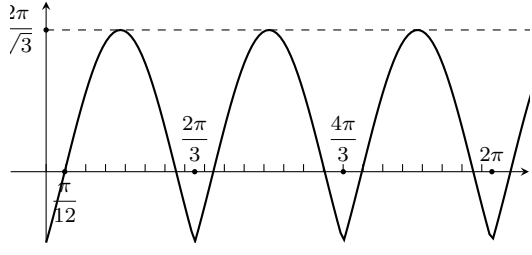
\begin{figure}[H]
\begin{center}

{
\begin{tikzpicture}[declare function={
    g(\x)=\x<2*pi/3 ? cos(2*(\x-pi/3) r) : 
          (\x< 4*pi/3 ? cos(2*(\x-pi) r):
      (\x< 2*pi ? cos(2*(\x-5*pi/3) r): 
      (\x<  9*pi/4 ? cos(2*(\x-pi/3) r): 
      cos(2*(\x-pi/3) r))));}]
    \begin{axis}[
      grid=both, 
      grid style={line width=0.1pt, draw=gray!75},
      axis lines=center,
      axis line style={black},
 unit vector ratio = 1 2,   
     xmin=-0.5, xmax=26*pi/12,
     ymin=-0.5, ymax=1.2,
      xtick=\empty,
      ytick=\empty,
      every axis plot/.append style={line width=1pt, color=black},
          ]
        \addplot[ thick, samples at={0,0.05,...,8}] {g(x)}; 
\fill(pi/12,0) circle (1pt) node[below] {\tiny$\displaystyle \frac\pi{12}$} ;
\fill(8*pi/12,0) circle (1pt) node[above] {\tiny$\displaystyle \frac{2\pi}{3}$} ;

\fill(4*pi/3,0) circle (1pt) node[above] {\tiny$\displaystyle \frac{4\pi}{3}$} ; 
\fill(2*pi,0) circle (1pt) node[above] {\tiny$\displaystyle {2\pi}$}; 
\fill(0,1) circle (1pt) node[left] {\tiny$\displaystyle {\frac{2\pi}{\sqrt 3}}$}; 
\draw[decorate,decoration={border, angle=90 , segment length=0.26179938779914943653855361 cm,
amplitude=0.1cm}] (0,0) - - (2*pi,0);
\draw [dashed](0,1) - - (8,1) ;
   \end{axis}
        \end{tikzpicture}}

\end{center}

   \caption{Illustration to Example 3: $y=h_{\Delta_3}(t). $ 
   }\label{Picture1}
   
\end{figure}

This function is 2-balanced. So, { $\sigma_Z(\Lambda)=\frac{2\pi}{\sqrt 3}$ }for every $2$-regular set $\Lambda$  with $\Delta_\Lambda=\Delta_3.$
Now take
$$k(t)=\dfrac12 h_{\Delta_3}(t+\dfrac\pi3).$$   
    Then (see Fig.~\ref{Picture2}) { \begin{gather*}
        \max_{t\in[0,2\pi]}(h_{\Delta_3}(t)+k(t))=
\frac{\pi}{\sqrt 3}\max_{t\in[0,\frac\pi3]}(2\cos2(t-\pi/3)+\cos 2t)\\=\frac{\pi}{\sqrt 3}\max_{t\in[0,\frac\pi3]}(\sqrt3\sin2t)=\pi.\end{gather*}
        Hence, for every 2-regular set $\Lambda$ with $\Delta_\Lambda=\Delta_3$ we have $$\sigma_U(\Lambda)\le\pi<\sigma_Z(\Lambda)=\frac{2\pi}{\sqrt3}.$$  }
    
     \begin{figure}[H]\label{graph2}
 \begin{center}
{
\begin{tikzpicture}[declare function={
    g(\x)=\x<2*pi/3 ? cos(2*(\x-pi/3) r) : 
      (\x< 4*pi/3 ? cos(2*(\x-pi) r):
  (\x<  2*pi ? cos(2*(\x-5*pi/3) r): 
      (\x<26*pi/12  ?cos(2*(\x-7*pi/3) r):
      cos(2*(\x-7*pi/3) r))));
    f(\x)=\x<pi/3 ? 0.5*cos(2*\x r) : 
   (\x<pi ? 0.5*cos(2*(\x-2*pi/3) r) :  
    (\x< 5*pi/3 ? 0.5*cos(2*(\x-4*pi/3) r): 
      0.5*cos(2*(\x-6*pi/3) r)
      )));}]
    \begin{axis}[
      grid=both, 
      grid style={line width=0.1pt, draw=gray!75},
      axis lines=center,
      axis line style={black},
     xmin=-0.8, xmax=26*pi/12,
     ymin=-0.2, ymax=0.6,
      xtick=\empty,
      ytick=\empty,
      every axis plot/.append style={line width=1.5pt, color=black},
          ]
       \addplot[dashed, thin, samples at={0,0.05,...,8}] {g(x)/2};  
        \addplot[dashed,thin,samples at={0,0.05,...,8}] {f(x)/2}; 
           \addplot[thick, samples at={0,0.05,...,8}] {f(x)/2+g(x)/2};
     
\fill(pi/3,0) circle (1pt) node[above] {\tiny$\displaystyle \frac{\pi}{3}$} ;
\fill(pi,0) circle (1pt) node[above] {\tiny$\displaystyle \pi$} ;
\fill(5*pi/3,0) circle (1pt) node[above] {\tiny$\displaystyle \frac{5\pi}{3}$} ; 
\fill(24*pi/12,0) circle (2pt) node[above] {\tiny$\displaystyle {2\pi}$}; 
\fill(0,1.732050807568877293527446341/4) circle (1pt) node[left] {\tiny$\displaystyle {\pi}$}; 
\fill(0,1/2) circle (1pt) node[left] {\tiny$\displaystyle {\frac{2\pi}{\sqrt3}}$};
\fill(0,1/4) circle (1pt) node[left] {\tiny$\displaystyle {\frac{\pi}{\sqrt3}}$};
\draw[decorate,decoration={border, angle=90 , segment length=0.26179938779914943653855361 cm,
amplitude=0.1cm}] (0,0) - - (2*pi,0);
\draw [dashed](0,1.732050807568877293527446341/4) - - (8,1.732050807568877293527446341/4) ;
\draw [dashed](0,1/2) - - (8,1/2) ;

   \end{axis}
        \end{tikzpicture}}  
         \end{center}
            \caption{Illustration to Example 3:  $y=h_{\Delta_3}(t)+k(t) $ (bold line)  }
            \label{Picture2}
\end{figure}
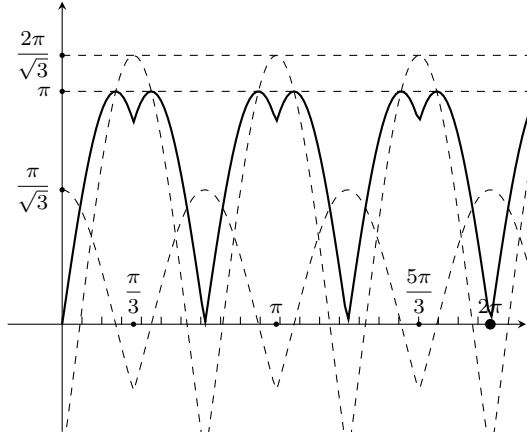

\paragraph{\bf Example 4: $ \rho\in (1,\infty)\setminus\mathbb N$}

 Let $\Delta_4=\delta_{-\frac{\pi}{2\rho}}+\delta_\frac{\pi}{2\rho}$, then
$$h_{\Delta_4}(t)=2\pi\cos(\rho t)\cdot \mathbf{1}_{[-\frac{\pi}{2\rho},\frac{\pi}{2\rho}]}.$$ 

Put $$k(t):=-\pi\cos \rho t \cdot \mathbf{1}_{[-\tau,\tau]},$$
where $\tau=\min(\frac{3\pi}{2\rho}, \pi)$.
Then { $$\max_{t\in[0,2\pi]}(h_{\Delta_4}(t)+k(t))=\pi<\max_{t\in[0,2\pi]}h_{\Delta_4}(t)=2\pi.$$}
        Hence, $\sigma_U(\Lambda)\le\pi<\sigma_Z(\Lambda)$ for every $\Lambda$ with $\Delta_\Lambda=\Delta_4$. 

   {\paragraph{\bf Example 5: $\rho\in(1/2,1)$.} 
   Let $\Delta_5=\delta_\pi$, then $h_{\Delta_5}(t)=\displaystyle\frac{\pi\cos \rho t}{\sin\pi \rho},\;\;t\in[-\pi,\pi].$
   
   Put
$$
k(t):=-\pi\cdot\frac{\cos\pi\rho}{\sin \pi\rho} \cdot\cos\rho(\pi-|t|),\;\;t\in[-\pi,\pi], $$
then
\begin{gather*}
    \max_{t\in[0,2\pi]}(h_{\Delta_5}(t)+k(t))=
\frac{\pi}{\sin\pi\rho}\cdot\max_{t\in[0,\pi]}\big(\cos \rho t-\cos\pi\rho\cdot\cos\rho(\pi-t)\big)=
\pi.
\end{gather*}}
        Hence, in this case we get $$\sigma_U(\Lambda)\le\pi<\sigma_Z(\Lambda)=\frac{\pi}{\sin\pi\rho}
        $$ for every $\Lambda$ with $\Delta_\Lambda=\Delta_5$. 

    As we will see later, all the estimates of $\sigma_U$ we obtained in this section for $\Delta_\Lambda=\Delta_n, n=1,2,3,4, 5$, occur  to be sharp. We return to these examples in Section~\ref{sec6}.

\section{Proof of Theorem \ref{T5}}

{\it (ii)$\Rightarrow$(i)}

From  the definition of a locally $\rho$-balanced set \eqref{loc-bal-geom} it follows that  set $M$ is locally $\rho$-balanced if and only if there exists a triple  of points $\{\alpha,\beta,\gamma\}\subset M$ such that
\begin{equation}\label{loc-bal}
0<\beta-\alpha\le\dfrac\pi\rho;\;\;0\le\gamma-\beta<\dfrac\pi\rho;\;\;\gamma-\alpha\ge\dfrac\pi\rho,
    \end{equation}
(here we use a simple observation that $0$ lies in the convex hull of a subset of the unit circle if and only if this subset contains   three points whose convex hull contains $0$).
In particular, in the degenerate 
 case when $\gamma=\beta=\alpha+\dfrac\pi\rho$ the set $M$ is locally $\rho$-balanced.

Note that for any function $k\in TC_\rho$ with $\rho>1/2$ {the set $\{ t:k(t)<0\}$ } is not locally $\rho$-balanced.

     Suppose that   {\it(ii)} does not hold, then there exists a function $k\in TC_{\rho}$ such that 
     \begin{equation}\label{TC}
     \max_{t\in[0,2\pi]} (\widehat h_\Lambda(t)+k(t))<\max_{t\in[0,2\pi]} \widehat h_\Lambda(t).
     \end{equation}
      Therefore, for every $t\in M_{\widehat h_\Lambda}$  we have
$k(t)<0$. Hence,  the function $\widehat h_\Lambda$ is not locally $\rho$-balanced.

{\it (i)$\Rightarrow$(ii)}

  { Note, that
    for the case $\rho=1$ the set is locally 1-balanced if and only if it is 1-balanced, so it suffices to consider the case $\rho\in(1/2,1)\cup(1,\infty)$ only.
}

       Suppose now that $\widehat h_\Lambda$ is not locally $\rho$-balanced. We will show how to construct a $\rho$-trigonometrically convex function $k$ such that $$\max_{t\in[0,2\pi]}\bigl(\widehat h_\Lambda(t)+k(t)\bigr)<\max_{t\in[0,2\pi]}\widehat h_\Lambda(t).$$     
       
      Note that a set $M$ {is not locally $\rho$-balanced}, if and only if it can be covered by a finite union of open disjoint intervals $M\subset \bigcup I_j,\;I_j \subset [0,2\pi ]$, such that the length of each $I_j$ is less than $\pi/\rho$, while the distance between  neighboring intervals is bigger than $\pi/\rho$.
      
 Put $M=M_{\widehat h_\Lambda}\subset[0,2\pi]$ and denote by $\widetilde M$ its $2\pi$-periodization.  Since the set $M$ is { not locally $\rho$-balanced,}
  there exists an open cover $\widetilde M\subset\bigcup_{j=1}^L I_j:=
\mathcal I, $ such that
\begin{itemize}
    \item for each $j$ 
    $$I_j=\bigcup_{k\in\mathbb Z}(\alpha_j+2\pi k,\beta_j+2\pi k),$$
    where $0<\beta_j-\alpha_j<\pi/\rho$;
    \item for each $j=1,\dots, L$
    $$\alpha_{j+1}-\beta_j>\pi/\rho,$$
    where $\alpha_{L+1}:=\alpha_1+2\pi.$
 
\end{itemize}    
    The supplementary  intervals to the set $\mathcal I$ we denote by      $J_j:=[\beta_j,\alpha_{j+1}]$.

       Put 
       $$a:=\max_{t\in[0,2\pi]}\widehat h_\Lambda(t);\;\;\;b:= \sup_{t\notin\mathcal I} \widehat h_\Lambda(t).$$
    From the definition of locally $\rho$-balanced set it follows that $b<a.$

   { 
Let us introduce an auxiliary function
$$s_\rho(t):=\sin\rho t\cdot \mathbf{1}_{[0,\frac\pi\rho]}.$$
Put
$$k(t)=\begin{cases}
   \dfrac{a-b}{2} \max\big(\sin\rho(t-\beta_j),\sin\rho(\alpha_j-t)\big),&t\in I_j;\\[10pt]
    \dfrac{a-b}{2} \max\big(s_\rho(t-\beta_j),s_\rho(\alpha_{j+1}-t)\big),&t\in J_j.
       
\end{cases}$$
Since $0<\beta_j-\alpha_j<\dfrac\pi\rho$, it follows that the function $k$ is {\it strictly  negative} on the set $\mathcal I$.   
On the other hand, $\displaystyle\max_{t\in[0,2\pi]}k(t)=\frac{a-b}{2}.$

 Recall \cite{Levin}, that a $2\pi$-periodic function $h$ is $\rho$-trigonometrically convex if and only if for all $ \alpha<\beta$ it holds
\begin{equation*}
    h'_+(\beta)-h'_-(\alpha)+\rho^2\int_\alpha^\beta h\ge0.
\end{equation*}
Since the function $k$ is continuous, $2\pi$-periodic and  piecewise $\rho$-trigonometric, to justify that $k\in TC_\rho$ it is sufficient to examine the set $\mathcal S$ of all singular points of  $k$ (that is, the points, where $k$ is not differentiable), and to show that for every $s\in\mathcal S$ it holds
\begin{equation}\label{s}
    k'_+(s)\ge k'_-(s),
\end{equation}

First of all, note that all the points $\alpha_j,\beta_j$ are regular points of the function $k$. 

Next, for each $j=1,\dots,L$ the point $\gamma_j:=\dfrac{\alpha_j+\beta_j}{2}$ is the only singular point on the interval $I_j,$ and since $k'_+(\gamma_j)>k'_-(\gamma_j),$ this function is $\rho$-trigonometrically convex on each $I_j$.

Further, if $|J_j|>\dfrac{2\pi}\rho,$ there are two singular points on $J_j$: $\beta_j+\dfrac{\pi}{\rho}$ and $\alpha_{j+1}-\dfrac{\pi}{\rho}$. It is easy to see that in this case the condition \eqref{s} also holds.

Finally, if $|J_j|\le\dfrac{2\pi}\rho,$ there is one singularity on the interval $ J_j$ at the point $\delta_j=\dfrac{\alpha_{j+1}+\beta_j}{2},$ and it holds $k'_+(\delta_j)>k'_-(\delta_j)$. 
Hence, the function $k$ is $\rho$-trigonometrically convex.


Thus, we have constructed a $\rho$-trigonometrically convex function $k$ such that
\begin{enumerate}
    \item[1.] if $t\in \mathcal I,$ then      $k(t)< 0.$
    Therefore,
    $$\widehat h_\Lambda(t) + k(t)<a;$$
\item[2.]
if $t\notin \mathcal I,$ then
$$\widehat h_\Lambda(t) + k(t)\le b+\frac{a-b}2=\frac{a+b}2<a.$$
\end{enumerate}
Hence 
$$\max_{t\in[0,2\pi]}\bigl(\widehat h_\Lambda(t)+k(t)\bigr)<\max_{t\in[0,2\pi]}\widehat h_\Lambda(t).$$     Theorem \ref{T5} is proved.
$\hfill\Box$
}
\paragraph {\bf Example 1 revisited.}
Note that all the functions $h_{\Delta^n_1}$ and $h_{\Delta^\infty_1}$ are { locally $\rho$-balanced  for $n\ge 2\rho$}, hence
$$\sigma_U(\Lambda_n)=\sigma_Z(\Lambda_n)={\frac{\pi }{n\sin\frac{\rho \pi}{n}}},$$
while $\sigma_U(\Lambda_\infty)=\sigma_Z(\Lambda_\infty)=\frac 1\rho.$

\section{Proof of Theorem \ref{T6}}\label{sec6}

We will prove inequalities
\eqref{6+} and \eqref{6ge}   by contradiction.

Suppose that $$\varepsilon:=A_\Lambda-\sigma_U(\Lambda)>0 .$$
Then there exists
 $k\in TC_\rho$ such that 
$$\max_{t\in[0,2\pi]} \bigl(h_\Lambda(t)+k(t)\bigr)<A_\Lambda-\varepsilon/2,$$
Using the fact that $
\displaystyle k(t)+k\bigl(t+\frac{\pi}{\rho}\bigr)\ge 0,$ we get 
$$h_\Lambda(t)+h_\Lambda\left(t+\frac{\pi}{\rho}\right)\le h_\Lambda(t)+h_\Lambda\left(t+\frac{\pi}{\rho}\right)+k(t)+k\left(t+\frac{\pi}{\rho}\right)< 2A_\Lambda-\varepsilon,$$
the derived contradiction proves the inequality \eqref{6+}.

 Suppose now that $\sigma_U(\Lambda)<C,$ then there exists $k\in {TC}_\rho$ such that 
 $$ \widehat{h}_\Lambda(t)+k(t)< C,\;\;\;\forall t.$$
 Therefore,
 $$ k(t)<0,\;\;\;\forall  t:\widehat{h}_\Lambda(t)\ge C.
 $$
Since $k$ is $\rho$-trigonometrically convex,  it follows that the set of its negative values can not contain a locally $\rho$-balanced set, leading us to a contradiction.

Theorem \ref{T6}  is proved.
$\hfill\Box$

\paragraph {\bf Examples 2-5 revisited.}

 \begin{enumerate}
     \item[\bf 2$'$.] Returning to the Example 2, we note that if $\Delta_\Lambda=\Delta_2$ for some $\rho$-regular set $\Lambda$, we have 
 $A_{\Lambda}=\pi$.  Now, due to inequality 
\eqref{6+}, we know that $\sigma_U(\Lambda)\ge\pi,$  and  since it has been shown in Example 2 that this level is achievable, we can conclude $$\sigma_U(\Lambda)=\pi.$$

\item[\bf 3$'$.]  In the same way we can find the value  of $\sigma_U(\Lambda)$ for the $\rho$-regular set $\Lambda $ with $\Delta_\Lambda=\Delta_3$, here we have {
\begin{gather*}
h_{\Delta_3}\left(\frac{5\pi}{12}\right)+h_{\Delta_3}\left(\frac{5\pi}{12}+\frac\pi 2\right)=\frac{4\pi}{\sqrt 3}\cos\frac{\pi}{6}=2\pi\le 
   2 A_\Lambda,\end{gather*}
hence,
$
\sigma_U(\Lambda)=\pi.$

The same result can be obtained by applying the second part of Theorem~\ref{T6}: put $C=\pi$ and consider the set 
$$\{t:h_{\Delta_3}(t)\ge \pi\}=\left[\frac\pi4,\frac{5\pi}{12}\right]\cup\left[\frac{11\pi}{12},\frac{13\pi}{12}\right]\cup \left[\frac{19\pi}{12},\frac{21\pi}{12}\right].$$
Since this set is locally $2$-balanced, we get by \eqref{6ge} that $\sigma_U(\Lambda)\ge \pi,$ and hence $\sigma_U(\Lambda)=\pi$.
}
\item[\bf 4$'$.]  Returning to the Example 4, we also see that 
 for the set $\Lambda $ with $\Delta_\Lambda=\Delta_4$, we have
   $A_\Lambda=\pi$, hence, $\sigma_U(\Lambda)=\pi.$
{\item[\bf 5$'$.]  For the case $\Delta_\Lambda=\Delta_5$  we have
 \begin{gather*}
     h_{\Delta_5}\left(\pi-\frac{\pi}{2\rho}\right)+h_{\Delta_5}\left((\pi-\frac{\pi}{2\rho})+\dfrac{\pi}{\rho}\right)\\
     =h_{\Delta_5}\left(\pi-\frac{\pi}{2\rho}\right)+h_{\Delta_5}\left(-\pi+\frac{\pi}{2\rho}\right)\\
     =  
     \frac{2\pi}{\sin\pi\rho}\cos(\pi\rho -\frac{\pi}{2}) =2\pi
 \le 2A_\Lambda.
 \end{gather*}
      Hence, $\sigma_U(\Lambda)=\pi.$}
      \end{enumerate}

 A natural question arises: do estimates \eqref{6+} and/or \eqref{6ge} always provide us an exact value for $\sigma_U(\Lambda)$? Unfortunately, this is not the case, as the following examples show.

{
\paragraph {\bf Example 2$''$.}
As has been shown in Example 2$'$, inequality \eqref{6+}
provides us with an exact bound on
$\sigma_U(\Lambda),$
while from the inequality \eqref{6ge}
 we cannot derive any meaningful bound, as the only value of
 $C\ge 0$ for which the set $\{t: h_{\Delta_2}(t)\ge C\}$ is locally $\rho$-balanced is $C=0.$

\paragraph{\bf Example 6.}
Let now $\rho=1$. Consider  some Reuleaux triangle $\mathcal R$ with 
width  $W=h(t)+h(t+\pi)$ and circumradius $R>W/2$. Suppose that its circumcenter is located at the origin. All three vertices of $\mathcal R$ lie on the circumcircle.

Let $h_6$  be the support function of $\mathcal R$, the corresponding measure we denote by $\Delta_6$. Then, if $\Delta_\Lambda=\Delta_6$, we have $\sigma_Z(\Lambda)=R$, and $A_\Lambda:= W/2$, so there is a strict inequality in \eqref{6+}, while the inequality  \eqref{6ge} gives us a sharp bound, since the convex hull of the vertices of the Reuleaux triangle contains the origin, hence the function $h_6$ is $1$-balanced, and for $\rho=1$ it follows that it is  locally $1$-balanced.

}



{\section{Proof of Theorem \ref{T7}}

Fix some $\Lambda \in \mathcal A_\rho$ and put $\Delta:=\Delta_\Lambda.$
Note that the conditions {\it (i)} and 
{\it (ii)} from the definition of $\mathcal A_\rho$ are equivalent to 
$$\sigma_U(\Lambda)=\sigma_Z(\Lambda)=1.$$

Hence, we have $\displaystyle\max_{t\in[0,2\pi]}\widehat h_\Lambda(t)=1, $ 
where, as before,  $\widehat h_\Lambda$ denotes  the  $\rho$-balanced modification of $h_\Delta$ defined by \eqref{lb-m}.

    The results in \cite[Chapter\ IV, Sect.~1]{Levin} (or just formula \eqref{h''}) give us that 
\begin{equation}\label{Dens}
    \mathcal D(\Lambda)=\frac{\rho}{2\pi}\int_{-\pi}^{\pi}\widehat h_\Delta(t)\;{\rm d}t,\end{equation}
 which immediately gives the upper bound:
$$\mathcal D(\Lambda)\le \rho=\frac{\rho}{2\pi}\int_{-\pi}^{\pi}\;{\rm d}t.$$

This estimate is achieved for any $\Lambda$ whose angular distribution is uniform, i.e.
$\Delta_\Lambda=\Delta^*([\alpha,\beta])=\dfrac{\rho}{2\pi}\cdot(\beta-\alpha)$
for all $0\le \alpha<\beta\le 2\pi$. The corresponding $\rho$-trigonometrically convex function is
$$h^*(t)=1.$$

Put $$M:=\{t\in\mathbb R:\widehat h_\Delta(t)=1\};$$
 this set is non-empty.

We now proceed to the proof of the lower bound, which is more involved. We subdivide it into three  parts depending on the value of $\rho$.

\begin{itemize}[leftmargin=*]
    \item{$\rho\in (0,1/2].$} 
       
Basic properties of  $\rho$-trigonometrically convex function imply that if $\theta_0\in M$, then for every $t,$ such that $|t-\theta_0|\le\frac\pi\rho$, we have
\begin{equation}
    \label{en}
   \widehat h_\Delta(t)\ge\cos\rho(t-\theta_0). 
\end{equation}
      Without loss of generality, we assume that $0\in M$.     
    Applying \eqref{en}, we get $$\int_{-\pi}^{\pi}\widehat h_\Delta(t){\rm d}t\ge\int_{-\pi}^{\pi}\cos\rho t\;{\rm d}t=\frac{2\sin\pi\rho}{\rho}.$$ 
    This bound is achieved for 
    $h_*(t)=\cos\rho t$, the corresponding measure is
    $$\Delta_*=\dfrac{\sin\pi\rho}{\pi}\cdot\delta_\pi.$$

   \item{ $\rho\in (1/2,1]$.} 

Since $\sigma_U(\Lambda)=\sigma_Z(\Lambda),$ by Theorem \ref{T5}, the set $M$ is locally $\rho$-balanced (see \eqref{loc-bal}). 
 Therefore,  without loss of generality, we can assume that there exists $\alpha\in \left[\pi-\dfrac\pi{2\rho}, \dfrac\pi{2\rho}\right],$ such that $\pm\alpha\in M$.
It follows that 
$$\widehat h_\Delta(t)\ge \cos\rho(|t|-\alpha),\;\;\;|t|\le\pi.$$

The auxiliary function 
\begin{gather*}
    f(\alpha)=\int_{-\pi}^\pi\cos\rho (|t|-\alpha)\; {\rm d}t=\frac2\rho\left(\sin\rho(\pi-\alpha)+\sin\alpha\rho\right)\\=
    \frac4\rho\sin\frac{\pi\rho}{2}\cos\left(\frac{\pi\rho}{2}-\alpha\rho\right)
    \end{gather*}
    is concave on the interval $\left[\pi-\dfrac\pi{2\rho},\dfrac\pi{2\rho}\right]$,  therefore its minimal value is attained on the boundary of this interval. Hence,
    $$f(\alpha)\ge \frac4\rho\sin^2\frac{\pi\rho}{2},$$
and therefore
$$\mathcal D(\Lambda)\ge \frac{\rho}{2\pi}\int_{-\pi}^{\pi}\widehat h_{\Delta}(t){\rm d}t\ge \frac{\rho}{2\pi} f(\alpha)\ge\frac{2}{\pi}\sin^2\frac{\pi\rho}{2}=\frac{1+|\cos\pi\rho|}{\pi}.$$

The bound is achieved for the function $h_*(t)=\sin\rho |t|,$ the corresponding angular density is given by
$$\Delta_*=\frac1\pi\left( \delta_0+|\cos\rho\pi|\cdot\delta_\pi\right).$$

\item{$\rho>1$.}{

Put 
$$N=[\rho-1/2]=\begin{cases}[\rho]-1,&\{\rho\}\in[0,1/2],\\
[\rho],&\{\rho\}\in(1/2,1),
\end{cases}$$
and $\delta:=\rho-N\in[1/2,3/2].
$

As in the previous case, since 
$\sigma_U(\Lambda)=\sigma_Z(\Lambda),$ by Theorem \ref{T5}, the set $M$ is locally $\rho$-balanced. Recall  that in this case at least one of the following conditions is satisfied (see \eqref{loc-bal}):
\begin{enumerate}
    \item $\exists\alpha: \;\;\left\{\alpha, \alpha+\dfrac{\pi}{\rho}\right\}\subset M;$
    \item $\exists \{\alpha,\beta,\gamma\}\subset M:\;
0<\beta-\alpha<\dfrac\pi\rho;\;\;0<\gamma-\beta<\dfrac\pi\rho;\;\;\gamma-\alpha>\dfrac\pi\rho.
$

\end{enumerate}

 }
Let us start with the case (i). Without loss of generality we suppose that $\alpha=-\dfrac\pi{2\rho}.$

Basic properties of  $\rho$-trigonometrically convex function  imply that if $\theta_0\in M$, then for every $t,$ such that $|t-\theta_0|\le\frac\pi\rho$, we have
\begin{equation}
   \widehat h_\Delta(t)\ge\cos\rho(t-\theta_0). 
\end{equation}
It follows that
$$\widehat h_\Delta(t)\ge \sin\rho|t|, \;\;\;|t|\le\frac{3\pi}{2\rho}.$$

 The next lemma follows immediately from standard properties of $\rho$-trigonometrically convex functions \cite[Chapter\ I, Sect.~16]{Levin}.
\begin{lemma}\label{int}
    If $h\in TC_\rho,$ then for each $x$ we have
$$\displaystyle\int_{x}^{x+2\pi/\rho}h\ge 0.$$
   
\end{lemma} 

Now, by Lemma \ref{int}, 
we get (recall that $\rho=N+\delta$)
\begin{gather*}
    \int_{-\pi}^{\pi}\widehat h_\Delta(t)\;{\rm d}t
=
\int_{-\frac{\delta\pi}{\rho}}^{2\pi\frac{N+\delta}{\rho}-\frac{\delta\pi}{\rho}}\widehat h_\Delta(t)\;{\rm d}t\ge
\int_{-\frac{\delta\pi}{\rho}}^{\frac{\delta\pi}{\rho}}\widehat h_\Delta(t)\;{\rm d}t
\\\ge
\int_{-\frac{\delta\pi}{\rho}}^{\frac{\delta\pi}{\rho}}\sin\rho|t|\;{\rm d}t=\frac2\rho(1-\cos\delta\pi).\end{gather*}
Recalling that $\rho-N=\delta\in[1/2,3/2],$  we obtain
\begin{equation}
    \label{type-i}
\int_{-\pi}^{\pi}\widehat h_\Delta(t)\;\mathrm{d}t
\ge\frac2\rho(1+|\cos\rho\pi|).\end{equation}

{
Proceeding with the case (ii), we introduce an auxiliary function
\begin{equation}
    \label{aux}
h_{\alpha,\beta,\gamma}(t):=
\begin{cases}
     \cos\rho(t-\alpha),&\alpha\le t\le\frac{\beta+\alpha}{2};\\
     \cos\rho(t-\beta),& \frac{\alpha+\beta}{2}\le t\le \frac{\beta+\gamma}{2};\\
\cos\rho(t-\gamma),&\frac{\beta+\gamma}{2}\le t\le\gamma;\\
\widehat h_\Delta,& \text{elsewhere}.
\end{cases}.
\end{equation}

Considering $\beta$ as a parameter, we will show that replacing $\beta$ by  $\beta^*=\alpha+\dfrac\pi\rho$ decreases the integral of the function $h_{\alpha,\beta,\gamma}$. 
This result will also be used in the next section, so we formulate it as a lemma.
\begin{lemma}
    \label{2-3}
   Given $ h_\Delta\in TC_\rho$ with $\displaystyle\max_{t\in[-\pi.\pi]} \widehat h_\Delta(t)=1$, let $\{\alpha,\beta,\gamma\}\subset M=\{t:\widehat h_\Lambda(t)=1\}$ be such that
$$\alpha<\gamma-\frac\pi\rho<\beta<\alpha+\frac\pi\rho<\gamma.$$ 
Put $\beta^*:=\alpha+\dfrac\pi\rho$. Then
    $$\int_{-\pi}^{\pi}h_{\alpha,\beta,\gamma}(t)\;{\rm d}t\ge \int_{-\pi}^{\pi}h_{\alpha,\beta^*,\gamma}(t)\;{\rm d}t.
    $$
\end{lemma}
\begin{proof}

Since 
$$h_{\alpha,\beta,\gamma}(t)=h_{\alpha,\beta^*,\gamma}(t)\;\;\forall t\in[-\pi,\pi]\setminus[\alpha,\gamma],$$
it is sufficient to consider an integral only on the interval $[\alpha,\gamma]$. We have
\begin{gather*}
\int_{\alpha}^{\gamma}h_{\alpha,\beta,\gamma}(t)\;{\rm d}t =\\=   \int_\alpha^{\frac{\beta+\alpha}{2}}\cos\rho(t-\alpha)\;{\rm d}t+
\int_{\frac{\beta+\alpha}{2}}^\frac{\gamma+\beta}{2}\cos\rho(t-\beta)\;{\rm d}t+
\int_{\frac{\gamma+\beta}{2}}^\gamma\cos\rho(t-\gamma)\;{\rm d}t\\
=
   2\int_{\frac{\beta+\alpha}{2}}^\frac{\gamma+\beta}{2}\cos\rho(t-\beta)\;{\rm d}t=
    \frac2\rho\left(\sin\rho\frac{\gamma-\beta}{2}-\sin\rho\frac{\alpha-\beta}{2}
\right)\\=
\frac4\rho\sin\rho\frac{\gamma-\alpha}{4}\cos\rho\frac{\alpha+\gamma-2\beta}{4}.\end{gather*}
Recalling the inequality
$$\alpha<\gamma-\frac\pi\rho<\beta<\alpha+\frac\pi\rho<\gamma,$$
we see that 
$$-\dfrac{\pi}\rho<\gamma-\alpha-\dfrac{2\pi}\rho<\alpha+\gamma-2\beta < \alpha-\gamma+\dfrac{2\pi}\rho<\dfrac{\pi}\rho.$$  Hence
$$\rho\dfrac{\alpha+\gamma-2\beta}{4}\in\left(-\dfrac\pi{2},\dfrac\pi{2}\right).$$ 
{ Therefore, by the elementary properties of the cosine function, we have 
$$\cos\rho\frac{\alpha+\gamma-2\beta}{4}>\cos\rho\frac{\alpha-\gamma+{2\pi}/{\rho}}{4}=\cos\rho\frac{\alpha+\gamma-2\beta^*}{4},$$
and finally}
$$\int_{-\pi}^{\pi}h_{\alpha,\beta,\gamma}(t)\;{\rm d}t\ge \int_{-\pi}^{\pi}h_{\alpha,\beta^*,\gamma}(t)\;{\rm d}t.$$
\end{proof}
Now, since the function $h_{\alpha,\beta^*,\gamma}$ is of type (i), applying \eqref{type-i} we get
$$\int_{-\pi}^{\pi}\widehat h_{\Delta}(t)\;{\rm d}t\ge\int_{-\pi}^{\pi}h_{\alpha,\beta,\gamma}(t)\;{\rm d}t\ge \int_{-\pi}^{\pi}h_{\alpha,\beta^*,\gamma}(t)\;{\rm d}t\ge\frac2\rho(1+|\cos\rho\pi|),$$
and therefore
$$\mathcal D\ge \frac{\rho}{2\pi}\int_{-\pi}^{\pi}\widehat h_{\Delta}(t){\rm d}t\ge \frac{1+|\cos\pi\rho|}{\pi}.$$
}

   To  show that this estimate is sharp, we introduce the following $\rho$-trigonometrically convex function:
$$k(t)=\begin{cases}
    \sin\rho |t|,&  |t|\le\frac{\delta\pi}\rho,\\
    \sin\rho t,&\frac{\delta\pi}\rho\le t\le2\pi-\frac{\delta\pi}\rho.\end{cases}
    $$

Since $k$ is piecewise $\rho$-trigonometric, to ensure that $k\in TC_\rho,$ it is sufficient to show that at the points $t=0, \pm\dfrac{\delta\pi}{\rho}$ we have 
$k'_+(t)\ge k'_-(t)$. Indeed, this inequality obviously holds at the points $0$ and $\dfrac{\delta\pi}\rho$, and at the point $-\dfrac{\delta\pi}\rho$ we have 
$$ {k}'_+\left(-\frac{\delta\pi}\rho\right)=-\rho\cos\delta\pi\ge0,$$
and
$$
{k}'_-\left(-\frac{\delta\pi}\rho\right)={k}'_-(2\pi-\frac{\delta\pi}\rho)=\rho\cos(2\pi\rho-\delta\pi)=
\rho\cos(\delta\pi)\le 0.$$
Hence, $k\in TC_\rho.$ Moreover, it is locally $\rho$-balanced and
\begin{gather*}
     \int_{-\pi}^{\pi}k(t){\rm d}t=
    \left( \int_{-\frac{\delta\pi}{\rho}}^{\frac{\delta\pi}{\rho}}+\int_{\frac{\delta\pi}{\rho}}^{\frac{\delta\pi}{\rho}+N\frac{2\pi}{\rho}}\right)k(t){\rm d}t
     \\
     =
     \int_{-\frac{\delta\pi}{\rho}}^{\frac{\delta\pi}{\rho}}k(t){\rm d}t
     =
     2\int_0^{\frac{\delta\pi}{\rho}} \sin\rho t\;{\rm d}t=\frac{2}{\rho}(1-\cos\delta\pi)=
     \frac{2}{\rho}\left(1+|\cos\rho\pi|\right).
          \end{gather*}

      { So, the estimate is achieved for $h_*=k$, and the corresponding angular density is
        $$\Delta_*=\frac1\pi\left(\delta_0+|\cos\rho\pi|\cdot\delta_{-\frac{\delta\pi}{\rho}}\right).$$
}

\end{itemize}

   Thus, in each case we have constructed two $\rho$-trigonometrically convex functions
$h_*$ and $h^*$, that correspond to the angular densities $\Delta_*$ and $\Delta^*$, such that  for any $\Lambda\in \mathcal A_\rho$ with $\Delta_\Lambda=\Delta_*$ the density map $\mathcal D$ attains its minimum value, and for any $\Lambda\in \mathcal A_\rho$  with $\Delta_\Lambda=\Delta^*$, the map $\mathcal D$ attains its  maximal value.

Now, for every $d\in[-1,1]$ put
\[
h_d:=\max\{h_*,d\}.
\]
Then $h_d$ is $\rho$-trigonometrically convex, $h_{-1}=h_*,$ $h_1=1=h^*$. The corresponding angular density we denote by $\Delta_d$. 
Then for any $\Lambda_d\in \mathcal A_\rho$ with $\Delta_{\Lambda_d}=\Delta_d$ we have
$$w(d):=\mathcal D(\Lambda_d)=\frac{\rho}{2\pi}\int_{-\pi}^\pi \widehat h_d(t) {\rm d}t=\frac{\rho}{2\pi}\int_{-\pi}^\pi \max\{\widehat h_*(t),d\} {\rm d}t.$$
The function $w$ depends continuously on $d$, hence, $\mathcal D(\mathcal A_\rho)$ fills the whole interval.

$\hfill\Box$


\section{Proof of Theorem \ref{ALS1}}

The proof of this theorem uses an idea from the proof of \cite[Theorem 2]{ALS}.

\begin{lemma}
    \label{U}
     If $\Lambda\in\mathcal B_\rho$ then $$\sigma_U(\Lambda)=\sigma_Z(\Lambda)=1.$$
\end{lemma}
\begin{proof}

Put   $A:=\sigma_U(\Lambda)\le\sigma_Z(\Lambda)=:B$, $\Delta:=\Delta_\Lambda.$
 
Recall that $\displaystyle B=\max_{t\in[-\pi,\pi]}\widehat h_\Delta(t)$, where   $\widehat h_\Delta$ is  the  $\rho$-balanced modification of $h_\Delta$ defined by \eqref{lb-m}.

The first condition of the definition of the set $\mathcal B_\rho$  implies that 
$A \le 1$.
In other words, there exists $k^*\in TC_\rho$, such that 
$$\widehat h_\Delta(t)+k^*(t)\le A\le1,\;\;\; \forall t\in[-\pi,\pi].$$  

Suppose that $A<1$. { Then for $\eps$ small enough, we have
$$\widehat h_\Delta(t)+ k^*(t)+\eps<1,\;\;\; \forall t\in[-\pi,\pi],$$  }{
that contradicts the second condition of the definition of the set $\mathcal B_\rho$ .
Therefore, $A= 1.$

Let us assume now that $k^*\in TC_\rho\setminus T_\rho$. Note that, in this case the corresponding density $\Delta^*$ is non-zero. Then, it follows from the second condition of the definition of the set $\mathcal B_\rho$,  that
$$\max_{t\in[-\pi,\pi]}(k^*(t)+\widehat h_\Delta(t))>1,$$
and we get a contradiction. 

Now, let $k^*\in T_\rho$. Since the function $\widehat h_\Delta$  is  $\rho$-balanced,  we obtain that
$$1=A\le B\le \max_{t\in[-\pi,\pi]}(k^*(t)+\widehat h_\Delta(t))\le 1.$$
}

The lemma is proved.

\end{proof}


Fix $\Lambda\in \mathcal B_\rho$ and put $\Delta:=\Delta_\Lambda$. From the conditions of the theorem, by Lemma \ref{U} and Theorem \ref{T5}, it follows that 
 $\widehat h_\Delta\in TC_\rho$ is a locally $\rho$-balanced function with $\displaystyle\max_{t\in[-\pi,\pi] }\widehat h_\Delta(t)=1.$ 
The upper bound follows immediately from \eqref{Dens}:
$$\mathcal D(\Lambda)\le \rho=\frac{\rho}{2\pi}\int_{-\pi}^{\pi}\;{\rm d}t.$$
The estimate is achieved for any $\Lambda$ with uniform angular density $\Delta^*$:
$$\Delta^*([\alpha,\beta])=\dfrac{\rho}{2\pi}\cdot(\beta-\alpha),\;\;\; 0\le \alpha<\beta\le 2\pi.$$
The corresponding $\rho$-trigonometrically convex function is
$h^*(t)=1.$

    Turning to the proof of the lower bound, let us start with the following lemma.
    

    \begin{lemma}\label{M-L}
      Let $\Lambda\in \mathcal B_\rho$ and $\Delta_\Lambda=\Delta$.   Then the set $$M=\{\theta:\widehat h_\Delta(\theta)=1\}$$ 
has no gap of  length greater than ${\pi}/{\rho}.$
\end{lemma}
\begin{proof}
First of all, from Lemma \ref{U} it follows that the set $M$ is non-empty. Hence, for $\rho\le 1/2$ the  lemma is trivial.

Suppose that $\rho>1/2$.
We use the fact that  the set $M$ is locally $\rho$-balanced.
Assume that $J:=\left[-\dfrac\pi{2\rho},\dfrac\pi{2\rho}\right]\subset \mathbb R\setminus M,$ and put
 $$\displaystyle m:=\max_{t\in J}\widehat h(t)<1.$$

Let $\Lambda_1$ be a $\rho$-regular set with angular density 
$\Delta_1$ that  corresponds to the $\rho$-trigonometrically convex function 
 $$h_{\Delta_1}(t)=\begin{cases}
    {(1-m)}\cos\rho t, &|t|\le\frac{\pi}{2\rho};\\
    0,&\frac{\pi}{2\rho}\le|t|\le\pi.
\end{cases}$$

Consider the  set $\Lambda'=\Lambda\cup \Lambda_1$. It is a $\rho$-regular set with  density 
$$\Delta'=\Delta+\Delta_1,$$
that corresponds to the $\rho$-trigonometrically convex function  $$h_{\Delta'}(t)=\widehat h_\Delta(t)+h_{\Delta_1}(t)\le 1.$$
Moreover, we have 
$$M'=\{\theta:h_{\Delta'}
(t)=1\}\supset M,$$ and hence $M'$ is locally $\rho$-balanced. Therefore, by Theorem~\ref{T5} and Corollary~\ref{Cor1}, 
 $\sigma_U(\Lambda')=\sigma_Z(\Lambda')=1.$ 
So, we have found a $\rho$-regular set with positive density, such that the set $\Lambda\cup\Lambda_1$ is a zero set for $\mathcal E_{\rho,1}$. This stands in contradiction to the second condition of the definition of $\mathcal B_\rho$, concluding the proof of the lemma.
\end{proof}


Returning to the proof of the theorem, let us start with the case $\rho\le1/2.$ 
\begin{itemize}[leftmargin=*]
    \item{$\rho\in (0,1/2].$} 

Without loss of generality, we assume that $0\in M.$
Then, applying inequality  \eqref{en},
 we get 
$$\mathcal D(\Lambda)\ge\frac\rho{2\pi}\int_{-\pi}^\pi \widehat h_\Delta\;{\rm d}t\ge\frac\rho{2\pi}\int_{-\pi}^\pi\cos\rho t \;{\rm d}t=\frac{1}{\pi}\sin\pi\rho, $$
and the minimum of the density is attained for $h_*(t)=\cos\rho t,$ which corresponds to the angular density $\Delta_*=\dfrac{\sin \pi\rho}{\pi}\delta_
\pi.$

{

 \item{$\rho>1/2.$} 

Without loss of generality we can assume that $0\in M$. From Lemma \ref{M-L} it follows that there exists a finite set $\Gamma:=\{\gamma_k\}_{k=0}^N\in M,$ where $$0=\gamma_0\le\gamma_1\le\ldots\le \gamma_N=2\pi,$$
such that
$0\le\gamma_k-\gamma_{k-1}\le {\pi}/{\rho}$.
Let us define a $\rho$-trigonometrically convex function $H$ as follows:   for $k=0,\dots, N-1$ we put
$$H_\Gamma(t)=\max(\cos\rho(t-\gamma_k),\cos\rho(t-\gamma_{k+1})),\;\;\;t:\;\gamma_{k}\le t\le \gamma_{k+1}.$$
Then, by \eqref{en}, we have $\widehat h_\Delta(t) \ge H_\Gamma(t).$

Applying Lemma \ref{2-3}, we observe that increasing the distance between neighboring points $\gamma_k$ until it equals  $\pi/\rho$ will decrease the integral  of the function $H_\Gamma$. 
Hence, the optimal configuration that minimizes the integral is achieved for $$\Gamma^*:=\{\gamma^*_k\}_{k=0}^{N^*}\subset M,$$ where $ N^*=[2\rho]$ and
$$\gamma_0^*=0,\gamma_1^*=\frac{\{2\rho\}\pi}\rho, \gamma_{k+1}^*=\gamma_k^*+\frac{\pi}{\rho},\;\;k=1,\ldots N^*-1.$$
It follows that
\begin{gather*}
    \int_{-\pi}^\pi H_\Gamma(t)\;{\rm d}t\ge \int_{-\pi}^\pi H_{\Gamma^*}(t)\;{\rm d}t\\=2\int_0^{\gamma_1^*}
    \cos \rho t \;{\rm d}t+2N^*\int_0^{\frac{\pi}{2\rho}}
    \cos \rho t \;{\rm d}t
    =\frac2\rho\left(\sin \frac{\{2\rho\}\pi}{2}+[2\rho]\right).
\end{gather*}
Finally, we get
$$\mathcal D(\Lambda)\ge\frac\rho{2\pi}\int_{-\pi}^\pi \widehat h_\Delta\;{\rm d}t\ge\frac\rho{2\pi}\int_{-\pi}^\pi H_\Gamma(t) \;{\rm d}t\ge 
\frac1\pi\left(\sin \frac{\{2\rho\}\pi}{2}+[2\rho]\right).$$

The measure, where the minimal value of the density is achieved, is
$$\Delta_*=\frac1\pi\left(\sin \frac{\{2\rho\}\pi}{2}\delta_{\frac{\{2\rho\}\pi}{2\rho}}+\sum_{k=1}^{[2\rho]-1}\delta_{\gamma^*_1+\frac{ (2k+1)\pi}{2\rho}}\right).$$

}\end{itemize}

The remainder of the proof is analogous to the final part of the proof of Theorem 1.7.
  For every $\rho>0$ we have constructed two angular densities $\Delta_*$ and $\Delta^*$, such that for any  $\Lambda\in \mathcal B_\rho$ with $\Delta_\Lambda=\Delta_*$ the density map $\mathcal D$ attains its minimum value on $\mathcal B_\rho$, and for any $\Lambda\in \mathcal B_\rho$  with $\Delta_\Lambda=\Delta^*$, the map $\mathcal D$ attains its  maximal value. The corresponding $\rho$-trigonometrically convex functions we denote by $h_*$ and $h^*$.
  
Put
\[
h_d:=\max\{h_*,d\}.
\]
Then $h_d$ is $\rho$-trigonometrically convex, $h_{-1}=h_*,$ $h_1=1=h^*$. The corresponding angular density we denote by $\Delta_d$. 
Then for any $\Lambda_d\in \mathcal B_\rho$ with $\Delta_{\Lambda_d}=\Delta_d$ we have
$$w(d):=\mathcal D(\Lambda_d)=\frac{\rho}{2\pi}\int_{-\pi}^\pi \widehat h_d(t) {\rm d}t=\frac{\rho}{2\pi}\int_{-\pi}^\pi \max\{\widehat h_*(t),d\} {\rm d}t.$$
The function $w$ depends continuously on $d$, hence, $\mathcal D(\mathcal B_\rho)$ is the interval $[\mathcal D(\Lambda_*),\mathcal D(\Lambda^*)]$.

}

\section{Application to random zero sets in Fock-type spaces}\label{Fock}

An entire function $f$  belongs to the {\bf Fock-type  space} $\mathcal F_\rho$, $\rho>0$,  if
\[\|f\|^2_{\rho}:=
\int_{\mathbb C}\left(\left|f(z)\right|
e^{-{|z|^\rho}}\right)^2 {\rm d}m(z)<\infty.
\]
Note that
$$\bigcup_{\sigma<1} \mathcal E_{\rho,\sigma}\subset \mathcal F_\rho\subset \mathcal E_{\rho,1}.$$

 The question we are concerned with in this section
is whether or not a kind of randomization of a given sequence of points is almost surely a zero set (or uniqueness set) of the Fock-type space $\mathcal F_\rho$.  

Let us fix a non-decreasing positive sequence $$\Lambda_{\mathbb R}=\{l_k>0\}_{k\in\mathbb N},$$ and suppose  that it has a finite density
$$\displaystyle \mathcal D=\lim_{R\to\infty}\frac{n_{\Lambda_{\mathbb R}}(R)}{R^\rho}<\infty.$$ 

To determine the type of randomization of $\Lambda_{\mathbb R}$, we fix a   probabilistic measure $\Delta$ on $[0,2\pi),$ and for $\rho\in\mathbb N$ we require in addition that it has zero $\rho$-th  moment
\begin{equation}\label{moment2}
    \int_0^{2\pi} e^{i\rho t}d\Delta(t)=0.
\end{equation}

{\it
       We define a {\bf$\Delta$-randomization} of $\Lambda_{\mathbb R}$ to be
    a random set $$\widetilde\Lambda_\Delta:=\{\lambda_k=l_ke^{i\theta_k}, l_k\in\Lambda_{\mathbb R}\},$$ where $\theta_k$ are independent random variables equally distributed with distribution~$\Delta$.  
}

The case of uniform randomization was considered earlier in  \cite{FT, AK}.

Note, that  each of the events 
"$\widetilde\Lambda_\Delta$
is a zero set of $\mathcal F_\rho$", 
"$\widetilde\Lambda_\Delta$
is a uniqueness set of $\mathcal F_\rho$" 
are tail events, and so by the Kolmogorov zero–one law each of them either almost surely happens or almost surely does not happen, for a fixed  sequence $\Lambda_{\mathbb R}$.

Let us show that the random set $\widetilde\Lambda_\Delta$ almost surely satisfies the Lindel\"of condition \eqref{Lind}.
Indeed, first of all, due to  condition~(\ref{moment2}),
$\mathbb E(\lambda_k^{-\rho})=0$ for each $k.$ Furthermore, $$\displaystyle\sum_{\lambda_k\in \widetilde\Lambda_\Delta}{\rm Var}(\lambda_k^{-\rho})=\sum_{l_k\in\Lambda_{\mathbb R}} {l_k^{-2\rho}}<\infty,$$ hence,
 by the Khinchine-Kolmogorov theorem  \cite[Theorem 22.6]{Bill}, (also known as Kolmogorov's two-series theorem), the series $\sum{\lambda_k^{-\rho}}$ converges almost surely. Thus, for the random set $\widetilde\Lambda_\Delta$ the Lindel\"of condition \eqref{Lind} holds almost surely.

 Our next step is to show that the set $\widetilde\Lambda_\Delta$  almost surely has angular density.

\begin{lemma}   A random sequence $\widetilde\Lambda_\Delta$  almost surely has the angular density~$\mathcal D\cdot\Delta$.
\end{lemma}

\begin{proof}
{ Even though this lemma may seem  obvious to specialists, we nevertheless prefer to prove it, relying on the following generalization of the Glivenko-Cantelli theorem, which is due to Varadarajan.}

\begin{ThV}\cite{Var}, \cite[Theorem 11.4.1]{Dudley}.
{
Given a separable metric space $(S,d)$ with a Borel $\sigma$-algebra $ \Sigma\subset 2^S$ and a probability measure $\mu: \Sigma\to [0,1]$, define the probability space $(\Omega, \mathbb P),$ where $\Omega :=S^\mathbb N $ and  $\mathbb P$ is the  corresponding  infinite product probability (uniquely defined,  \cite[Theorem 8.2.2]{Dudley}).
 For $\omega=(\xi_1^{\omega},\xi_2^{\omega},\cdots)\in\Omega$ define empirical distribution $\mu_n^{\omega}$ on $\sigma$-algebra $\Sigma$ by $$\mu_n^{\omega}:=\frac1n\left(\delta_{\xi^{\omega}_1}+\delta_{\xi^{\omega}_2}+\dots+\delta_{\xi^{\omega}_n}\right).$$
Then $\mu_n^{\omega
}$ weakly converges to $\mu$ almost surely:
$$\mathbb P(\{\omega\in S^{\mathbb N}: \mu_n^{\omega
}\xrightarrow{weak}\mu\})=1.$$
}
\end{ThV}

The weak convergence of empirical distributions ${\mu_n^{\omega}}$  to $\mu$ implies that for all arcs $(\alpha,\beta)$ such that $\mu(\{\alpha\})=\mu(\{\beta\})=0,$ we have 
$$\lim_{n\to\infty}{\mu_n^{\omega}}(\alpha,\beta)=\mu(\alpha,\beta),$$
where $\mu$ is some probability measure on a metric space $S$, and $\mu_n^{\omega}$ are corresponding empirical measures.

Now, if we take $S=[0,2\pi)$, $\mu=\Delta,$ and $\xi_k=\theta_k,$ then for all $\alpha,\beta\in[0,2\pi)$, and for any $R\in (\lambda_n,\lambda_{n+1}],$$$\mu_n^{\omega}(\alpha,\beta)=\frac{n_\Lambda(R;\alpha,\beta)}{n}=\frac{n_\Lambda(R;\alpha,\beta)}{n_{\Lambda_{\mathbb R}}(R)}.$$
Hence, {with probability one, for all but a countable set of angles}

$$\lim_{R\to\infty}\frac{n_\Lambda(R;\alpha,\beta)}{R^\rho}=\lim_{R\to\infty}\frac{\mu_{n(R)}^{\omega}(\alpha,\beta) {n_{\Lambda_{\mathbb R}}(R)}}{R^\rho}=\mathcal D\cdot\Delta(\alpha,\beta).
$$

\end{proof}

It follows that a random sequence $\widetilde\Lambda_\Delta$ is almost surely $\rho$-regular, hence  the results of previous sections can be applied. 
Thus, all theorems lead to corollaries that hold almost surely for the random sequences  $\widetilde\Lambda_\Delta$.
In particular, the following corollary is true.
\begin{corollary}
 Let  $\rho\in\mathbb N$, and  let a probabilistic measure $\Delta$ on $[0,2\pi)$ have zero $\rho$-th moment.
Given a positive sequence $\Lambda_{\mathbb R}$   with density $\mathcal D$,  its randomization $\widetilde \Lambda_{\Delta}$ is almost surely
\begin{itemize}

    \item[-] a zero set of $\mathcal F_\rho$, in case $\mathcal DR_\Delta^*<1;$
      \item[-] not a zero set of $\mathcal F_\rho$ in case
    $\mathcal DR_\Delta^*>1.$
    \end{itemize}       
\end{corollary}
The critical case $\mathcal DR_\Delta^*=1$ here is more subtle and requires a special consideration. For the uniform randomization  some examples of a.s. zero sets of critical density and a.s. uniqueness sets of critical density for $\mathcal F_2$ can be found in \cite{AK}.

{
Now, shifting focus to the difference between $\sigma_U$ and $\sigma_Z$ and referring to Theorem \ref{T4} and Examples 2-4, we arrive at the following corollary.

\begin{corollary}
For every $\rho\in(1/2,\infty)\setminus\{1\}$, there exists a probability measure $\Delta$ satisfying condition \eqref{moment}, and a  nonempty open interval $J_{\rho,\Delta}$ such that for any positive  sequence $\Lambda_{\mathbb R}$ with density $\mathcal D\in J_{\rho,\Delta}$, its  randomization $\widetilde \Lambda_\Delta$ is almost surely a non-zero set and a non-uniqueness set for the Fock space $\mathcal F_\rho.$
\end{corollary}

\begin{samepage}
\paragraph{\bf Example 7.}
Put $\Lambda_{\mathbb R}=\{\sqrt \frac nD\},$ and let $\Delta=\frac13 \Delta_3,$ where $\Delta_3$ is defined as in Example 3. Then the random set $\widetilde \Lambda_\Delta$ is almost surely
\begin {itemize}
\item a zero set of $\mathcal F_2$ if
$D<\dfrac{3\sqrt 3}{2\pi};$
\item a non-zero and non-uniqueness set of $\mathcal F_2$ if
$\dfrac{3\sqrt 3}{2\pi}<D<\dfrac{3}{\pi};$
\nopagebreak
\item a uniqueness set of $\mathcal F_2$ if
$D>\dfrac{3}{\pi}.$

\end{itemize}
\end{samepage}
}

\subsection*{Acknowledgements}
The author is deeply grateful to Alexander Borichev and Mikhail Sodin for their kind support, guidance in structuring the material and insightful comments; to  Evgeny Abakumov for 	valuable discussions and suggestions on the topic; and to the anonymous reviewers for careful reading of the manuscript and constructive feedback.	

The author expresses sincere gratitude for the support provided  by Israel Science Foundation grant  No. 1288/21, and by the Center for Integration in Science of the Israel's Ministry of Aliyah and Integration.

The author also sincerely thanks Tel Aviv University for the inspiring and genuinely supportive academic atmosphere during her time there, which played an important role in the development of this work.

\bigskip

\bigskip
\medskip

\end{document}